\documentclass[11pt]{article}
\usepackage{mathrsfs}
\usepackage{mathrsfs}
\usepackage{mathrsfs}
\usepackage{diagbox}
\usepackage{cite}
\usepackage{amsfonts}
\usepackage{amssymb}
\usepackage{amsmath}
\usepackage{amsthm}
\usepackage{dsfont}
\usepackage[colorlinks,
            linkcolor=blue,
            anchorcolor=blue,
            citecolor=blue
            ]{hyperref}
\theoremstyle{plain}
\newtheorem{thm}{Theorem}[section]
\newtheorem{cor}[thm]{Corollary}
\newtheorem{lem}[thm]{Lemma}

\theoremstyle{definition}
\newtheorem{defn}{Definition}[section]

\newtheorem{Assump}{Assumption}[section]
\theoremstyle{remark}
\newtheorem{rem}{Remark}[section]

\begin{document}
\title{{\Large\bf {$L^p$-solutions of QSDE driven by Fermion fields with nonlocal conditions under non-Lipschitz coefficients}}
}
\author{{\normalsize Guangdong Jing,\thanks{Department of Mathematics, 
Beijing Institute of Technology,    Beijing 100081, China.} \quad  Penghui Wang,\thanks{School of Mathematics, Shandong University, Jinan, 250100, China.} \quad  Shan Wang,\thanks{School of Mathematics, Shandong University, Jinan, 250100, China.} \thanks{E-mail addresses: 201511240@mail.sdu.edu.cn(G.Jing), phwang@sdu.edu.cn(P.Wang), 202020244@mail.sdu.edu.cn(S.Wang).}   }   }
\date{}
\maketitle
\begin{minipage}{14cm} {\bf Abstract.}
The main purpose of this paper is to obtain the existence and uniqueness of $L^p$-solution to quantum stochastic differential equation
driven by Fermion fields with nonlocal conditions in the case of non-Lipschitz coefficients for $p>2$. 
The key to our technique is to make use of the Burkholder-Gundy inequality given by Pisier and Xu and Minkowski-type inequality to iterate instead of the fixed point theorem commonly used for nonlocal problems.
Moreover, we also obtain the self-adjointness of the $L^p$-solution which is important in the study of optimal control problems.


\end{minipage}
 \maketitle
\numberwithin{equation}{section}
\newtheorem{theorem}{Theorem}[section]
\newtheorem{lemma}[theorem]{Lemma}
\newtheorem{proposition}[theorem]{Proposition}
\newtheorem{corollary}[theorem]{Corollary}
\newtheorem{remark}[theorem]{Remark}


\section{Introduction}
\indent\indent
In the present paper, we shall consider the following quantum stochastic differential equation (QSDE for short) introduced by Barnett, Streater and Wilde \cite{BSW2}:
\begin{equation}\label{SDE}
dX_t=F(X_t,t)dW_t+dW_tG(X_t,t)+H(X_t,t)dt,
\end{equation}
driven by Fermion fields which is closely related to the quantum noise, quantum fields etc. 
The $L^2$-solution of the QSDE \eqref{SDE} under Lipschitz coefficients was studied by Barnett, Streater and Wilde \cite{BSW2,BSW3} and  Applebauma-Hudson-Lindsay \cite{AH1,AH2,LJ} using It\^{o}-isometry and It\^{o}-formula, respectively. Fermion fields and Boson fields are the two most important quantum fields. The QSDE driven by Boson fields have also been studied extensively \cite{BSW4,HL,HP,LW}. Moreover, these two classes of QSDEs can be unified understood as a same framework of the Hudson and Parthasarathy's quantum stochastic calculus\cite{RL,K} in noncommutative spaces.

The noncommutative $L^p$-space and associated harmonic analysis have been deeply studied in \cite{CK,HM,SIE1,SIE0,SIE,SIE2} and references therein. In particular,
martingale inequalities in the non-commutative case have been greatly developed since Pisier and Xu \cite{PX} proved the analogy of the classical Burkholder-Gundy inequality for non-commutative martingales, where they defined the noncommutative $H^p$ martingale space by introducing the row and column square functions.
Subsequently, combined Khintchine inequalities of operator values obtained by Lust-Piquard \cite{LF}, Junge, Randrianantoanina and Xu et al generalized almost all the martingale inequalities such as Doob maximal inequality and Rosenthal inequalities so on of classical martingale theory to the noncommutative case \cite{J,JX1,JX2,JX3,JX4,RN1,RN2}, which lays a foundation for the study of quantum stochastic analysis.


Inspired by Burkholder-Gundy inequality, it is natural to solve the QSDE \eqref{SDE} in non-commutative $L^p$-spaces. Problems with nonlocal conditions are more applicable to real life problems than problems with traditional local conditions. Analogously, we study the $L^p$-solution and the basic properties of solution to QSDE with nonlocal initial conditions for $p>2$.
To state our results, we give some assumptions on the QSDE. Some basic notations of noncommutative filtration $\{\mathscr C_t\}_{t\geq0}$ will be introduced in Section 2.

\begin{defn}
A map $X:~\mathbb{R}^+\rightarrow L^p(\mathscr{C})$ is said to be adapted if $X_t\in L^p(\mathscr{C}_t)$ for each $t\in \mathbb{R}^+$. A map $F:~L^p(\mathscr{C})\times\mathbb{R}^+\rightarrow L^p(\mathscr{C})$ is said to be adapted if $F(u, t)\in L^p(\mathscr{C}_t)$, for any $t\in \mathbb{R}^+$ and $u\in L^p(\mathscr{C}_t)$.
\end{defn}

It is easy to check that if $X: \mathbb{R}^+\rightarrow L^p(\mathscr{C})$ and $F:L^p(\mathscr{C})\times\mathbb{R}^+\rightarrow L^p(\mathscr{C})$ are both adapted, so is the map $t\mapsto F(X_t,t)$.

In the rest of this section, we consider the following equation with nonlocal conditions in the interval $[0, T]$ for some fixed $T > 0$,
\begin{equation}\label{nonlocal condition-QSDE}
\left\{
\begin{aligned}
  dX_t&=F(X_t,t)dW_t+dW_tG(X_t,t)+H(X_t,t)dt,\ {\rm in}\ [t_0,T],\\
  X_{t_0}&=Z+R(X),
\end{aligned}
\right.
\end{equation}
where $F(\cdot,\cdot),\ G(\cdot,\cdot),\ H(\cdot,\cdot):L^p(\mathscr{C})\times[0,T]\rightarrow L^p(\mathscr{C})$, $R(\cdot):L^p(\mathscr{C})\to L^p(\mathscr{C})$ constitutes the nonlocal condition, and $X_{t_0}\in L^p(\mathscr{C}_{t_0})$ for fixed $p>2$.
\begin{defn}
A stochastic process $X(\cdot):[0,T]\to L^p(\mathscr{C})$ is called a solution to the equation \eqref{nonlocal condition-QSDE} if it satisfies
$$X_t=Z+R(X_t)+\int_{t_0}^{t}F(X_s,s)dW_s+\int_{t_0}^{t}dW_s G(X_s,s)+\int_{t_0}^{t}H(X_s,s)ds,\  \textrm{a.s.\ in}\ [t_0,T].$$
\end{defn}
Throughout the paper, we shall make use of the following conditions:
\begin{Assump}\label{Assump}
$F(\cdot,\cdot),\ G(\cdot,\cdot),\ H(\cdot,\cdot)$: $L^p(\mathscr{C})\times[0, T]\rightarrow L^p(\mathscr{C})$ are operator-valued functions such that
\begin{description}
  \item[(A1)] $F(\cdot,\cdot),\ G(\cdot,\cdot),\ H(\cdot,\cdot):L^p(\mathscr{C})\times\mathbb{R}^+\rightarrow L^p(\mathscr{C})$ are adapted.
  \item[(A2)] For any $x\in L^p(\mathscr{C})$, $F(x,\cdot),\ G(x,\cdot),\ H(x,\cdot):[0,T]\to L^p(\mathscr{C})$ are continuous a.s.
 \item[(A3)] Osgood condition:\  For any $x_1,\ x_2\in L^p(\mathscr{C})$ and a.e. $t\in [0,T]$,
$$
   \|F(x_1, t)-F(x_2, t)\|_p^2+\|G(x_1, t)-G(x_2, t)\|_p^2+ \|H(x_1, t)-H(x_2, t)\|_p^2\leq \rho(\|x_1-x_2\|_p^2) .\\
$$
where $\rho: \mathbb{R}^+\to\mathbb{R}^+$ is a continuous non-decreasing function with $\rho(0)=0$, $\rho(r)>0$ for $r>0$, such that $\int_{0^+}\frac{dr}{\rho(r)}=+\infty$.
  \item[(A4)] $R$ is continuous and adapted, and there exists a constant $0<C(R)<1$ such that
\begin{equation}\label{J-Lipschitz}
\|R(x_1)-R(x_2)\|_p\leq C(R)\|x_1-x_2\|_p, \  \forall\ x_1,\ x_2\in L^p(\mathscr{C}).
\end{equation}
\end{description}
\end{Assump}
\begin{thm}Under the assumptions (A1)-(A4), for $p>2$, there is a unique continuous adapted $L^p$-solution
$\{X_t\}_{t\geq t_0 }$ to the following quantum stochastic integral equation with nonlocal initial value
\begin{equation}\label{Iteration}
X_t=Z+R(X_t)+\int_{t_0}^{t}F(X_s,s)dW_s+\int_{t_0}^{t}dW_s G(X_s,s)+\int_{t_0}^{t}H(X_s,s)ds,\ {\rm{in}}\ [t_0, T ].
\end{equation}
\end{thm}


\begin{rem}
There are many efforts to study the solution to the QSDE (\ref{SDE}) by many mathematicians.
\begin{itemize}
\item[(1).] By using It\^{o}-isometry, Barnett, Streater and Wilde \cite{BSW2,BSW3} considered the $L^2$-solution to the QSDE \eqref{SDE}.
\item[(2).] The original study \cite{AH1} on the following QSDE 
 \begin{equation}\label{SDE-2}
 dX_t=F(X_t,t)dA_t+dA^*_t G(X_t,t)+H(X_t,t)dt,\ \textrm{in}\ [0,T],
\end{equation}
 is to consider the solution in the weak sense, i.e., $X_t$ is called the solution to the initial value problem of the QSDE, if
$$
X_t u=\left(Z+ \int_{0}^t F(X_s,s)dA_s+\int_0^t dA^*_s G(X_s,s)+\int_0^t H(X_s,s)ds\right)u,
$$
for any $u\in D(X_t)$, where $X_t$ can be considered as an unbounded operator densely defined on the Hilbert space $\Lambda(L^2(\mathbb R^+))$ with the domain $D(X_t)$.
\item[(3).] The martingale inequality and Burkholder-Gundy inequality of Pisier and Xu have been used by Dirksen \cite{D} to study the $L^p$-solution for $p>2$ to QSDE with respect to any normal $L^p$-martingale without the drift term $\int_0^t H(X_s,s)ds$. The reason why the Burkholder-Gundy inequality can be used directly is because, without the drift term $\int_0^t H(X_s,s)ds$, the solution of the QSDE is a martingale, and hence by the Lipschiz condition on ${\cal H}^p([0,t])$, one can have existence and uniqueness of the solution.
\item[(4)] Compared with the result of Dirksen \cite{D}, the QSDE \eqref{SDE} that we consider has the drift term. The reason why we can deal with the drift term is because we obtain the estimation in $L^2(0,T; L^p({\mathscr C}_T))$ with the help of Burkholder-Gundy inequality and Minkowski inequality.
\end{itemize}
\end{rem}



This paper is organized as follows. 
In Section 2, we review some fundamental notations and preliminaries on Fermion fields, and introduce several useful inequalities which are the main techniques of subsequent proof.
In section 3 and section 4, we obtain the existence, uniqueness and the stability of $L^p$-solution to the QSDE \eqref{SDE} and \eqref{SDE-2} with nonlocal conditions.
In section 5, we get the self-adjointness and the Markov property of the solution to QSDE.

\section{Preliminaries and Burkholder-Gundy inequality}
\indent\indent
In this section, we introduce the main techniques to solve problems later.
We first recall some concepts \cite{BSW1,BSW2,BSW3,PX,W,W2} necessary to the whole article.

Let $\mathscr{H}$ be a separable complex Hilbert space.
For any $z\in\mathscr{H}$, the creation operator $C(z):\Lambda_n(\mathscr{H})\rightarrow \Lambda_{n+1}(\mathscr{H})$ defined by $u\mapsto \sqrt{n+1}\ z\wedge$$u$, is a bounded operator on $\Lambda(\mathscr{H})$ with $\|C(z)\|=\|z\|$, where $\Lambda_0(\mathscr{H}):=\mathbb{C}$.
Meanwhile, define the annihilation operator $A(z)=C^*(z)$.
Then the antisymmetric Fock space $\Lambda(\mathscr{H}):=\oplus_{n=0}^\infty\Lambda_n(\mathscr{H})$ over $\mathscr{H}$ is a Hilbert space.
Moreover, define the fermion field $\Psi(z)$ on $\Lambda(\mathscr{H})$ by $\Psi(z):=C(z)+A(Jz)$, where $J: \mathscr{H}\rightarrow\mathscr{H}$ is the conjugation operator (i.e., $J$ is antilinear, antiunitary and $J^2=1$).
Denote by $\mathscr{C}$ the von Neumann algebra generated by the bounded operators $\{\Psi(z): z\in \mathscr{H}\}$. 
For the Fock vacauum $\mathds{1}\in\Lambda(\mathscr{H})$, define
\begin{equation}\label{m-state}
m(\cdot):=\langle\mathds{1}, \cdot\mathds{1}\rangle
\end{equation}
on $\mathscr{C}$. Obviously, $m$ is a normal faithful state on $\mathscr{C}$, and $(\Lambda(\mathscr{H}), \mathscr{C}, m)$ is a quantum probability space by \cite{MT}.

Now, let $\mathscr{H}=L^2(\mathbb R^+)$, and  $\mathscr{C}_t$ be the von Neumann subalgebra of $\mathscr{C}$ generated by $\{\Psi(u):u\in L^2(\mathbb R_+), {\rm~ess~supp}~u\subseteq [0,t]\},$
then $\{\mathscr{C}_t\}_{t\geq 0}$ is an increasing family of von Neumann subalgebras of $\mathscr{C}$ which is the noncommutative analogue of filtration in stochastic analysis.
Let$$W_t=\Psi(\chi_{[0,t]})=C(\chi_{[0,t]})+A(J\chi_{[0,t]}),\  t\geq 0,$$ then $\{W_t:t\in \mathbb{R}^+\}$ is a Fermion Brownian motion adapted to the family $\{\mathscr{C}_t: t\in \mathbb{R}^+\}$.
For any $1\leq p<\infty$, define the noncommutative $L^p$-norm on $\mathscr{C}$ by
 $$\|f\|_{p}:=m(|f|^p)^{1\over p}=\langle\mathds{1}, |f|^p\mathds{1}\rangle^{1\over p}$$
where $|f|=(f^*f)^{1\over 2}$, then $L^p(\mathscr{C},m)$ is the completion of $(\mathscr{C}, \|\cdot\|_p)$, which is the noncommutative $L^p$-space, abbreviated as $L^p(\mathscr{C})$. 
For any interval $[0,T]\subset\mathbb{R^+}$, set
\begin{align*}
 C_{\mathbb{A}}(0,T;L^p(\mathscr{C}_T)):=\Big\{x(\cdot):\ & [0,T]\to L^p(\mathscr{C}_T) \mid x(t)\in L^p(\mathscr{C}_t)\ \\
   & {\rm and}\ \lim_{s\to t}\|x(s)-x(t)\|_p=0,\ \forall\ s,\ t\in[0,T]\Big\}.
\end{align*}
It is easy to check that $C_{\mathbb{A}}(0,T;L^p(\mathscr{C}_T))$ is a Banach space with the norm given by
$$\|x\|_{C_{\mathbb{A}}(0,T;L^p(\mathscr{C}_T))}=\max_{t\in[0,T]}\|x(t)\|_p.$$


\begin{defn} An adapted $L^p$-processes $f$ on $[t_0,t]$ is said to be simple if it can be expressed as
\begin{equation}\label{Simple}
f=\sum_{k=0}^{n-1}f(t_k)\chi_{[t_{k}, t_{k+1})}
\end{equation}
on $[t_0,t]$ 
for $t_0\leq t_1\leq \cdots \leq t_n=t$ and $f(t_k)\in L^p(\mathscr{C}_{t_k})$ for all $0\leq k\leq n-1$.
\end{defn}

\noindent By \cite{BSW1}, 
the It\^{o}-Clifford integral of any simple adapted $L^p$-process with respect to Fermion  Brownian motion $W_t$ is defined as follows.
\begin{defn}
If $f=\sum\limits_{k}f(t_k)\chi_{[t_{k}, t_{k+1})}$ is a simple adapted $L^p$-processes on $[t_0,t]$, then the
It\^{o}-Clifford stochastic integral of $f$ over $[t_0, t]$ with respect to $W_t$ is
\begin{equation}\label{Right-stochatic integral}
\int_{t_0}^tf(s)dW_s=\sum_{k=0}^{n-1}f(t_k)(W_{t_{k+1}}-W_{t_{k}}).
\end{equation}
\end{defn}

For $p\geq 1$, let ${\cal S}^p_{\mathbb A}(\mathbb R^+)$ be the linear space of all simple adapted $L^p$-processes, i.e.
$$
{\cal S}^p_{\mathbb A}(\mathbb R^+):=\{f:\mathbb R^+\to L^p (\mathscr{C}),\ f\text{ is  simple and adapted}\}.
$$
Then ${\cal S}^p_{\mathbb A}([0,t])$ is subspace of ${\cal S}^p_{\mathbb A}(\mathbb R^+)$ whose processes vanish in $(t,\infty)$.
It is clear that $\int_{t_0}^tf(s)dW_s$ is a Clifford $L^p$-martingale for any $f\in {\cal S}^p_{\mathbb A}([0,t])$, i.e. $\mathbb{E}\left(\int_{t_0}^tf(\tau)dW_\tau\mid \mathscr{C}_s\right)=\int_{t_0}^sf(\tau)dW_\tau$ for any $t_0\leq s\leq t$. 
 For all $f\in {\cal S}^p_{\mathbb A}([0,t])$, let
\begin{equation*}
\left\|f\right\|_{\mathcal{H}^p([0,t])}:=\max\left\{\left\|\left(\int_0^t|f(s)|^2ds\right)^{\frac{1}{2}}\right\|_p ,\ \left\|\left(\int_0^t|f^*(s)|^2ds\right)^{\frac{1}{2}}\right\|_p\right\},
\end{equation*}
and denote by the non-commutative Hardy space $\mathcal{H}^p([0,t])$ the completion of ${\cal S}^p_{\mathbb A}([0,t])$ under the norm $\|\cdot\|_{\mathcal{H}^p([0,t])}$. For simplicity, we use $\mathcal{H}^p_{loc}(\mathbb{R}^+)$ to represent the set of stochastic processes $f: \mathbb{R}^+\to L^p(\mathscr{C})$ and $\chi_{[0,t]}f\in \mathcal{H}^p([0,t])$.
Moreover, we have the following the Burkholder-Gundy inequality \eqref{BDG-PX} of Clifford $L^p$-martingale first established in \cite{PX}.

\begin{lem}\cite[Theorem 4.1]{PX}\label{px_bdg}
Let $2\leq p<\infty$. Then, for any $f\in \mathcal{H}^p_{loc}(\mathbb{R}^+)$ and its It\^{o}-Clifford integral
\begin{equation*}
X_t=\int_0^tf(s)dW_s, \ t\geq 0,
\end{equation*}
it holds that
\begin{equation}\label{BDG-PX}
\alpha_p^{-1}\|f\|_{\mathcal{H}^{p}([0,t])}\leq\|X_t\|_p\leq\beta_p\|f\|_{\mathcal{H}^{p}([0,t])}, \ t\geq0,
\end{equation}
where $\alpha_p$ and $\beta_p$ are constant depend on $p$.
\end{lem}
The stochastic integral \eqref{Right-stochatic integral} is also called right stochastic integral. Analogously, we can define left stochastic integrals, and have the Burkholder-Gundy inequality with respect to left stochastic integrals. 
\begin{lem}\cite[Theorem 7.2]{D}\label{L-R}
Let $1\leq p< \infty$. For any $f\in \mathcal{H}^p_{loc}(\mathbb{R}^+)$,  then the left stochastic integral $\int_0^t dW_s f(s)$ and right stochastic integral $\int_0^tf(s)dW_s$ can be well-defined and
$$\left\|\int_0^t dW_s f(s)\right\|_p\simeq_p\|f\|_{\mathcal{H}^{p}([0,t])}\simeq_p\left\|\int_0^t f(s)dW_s\right\|_p.$$
\end{lem}

By Lemma \ref{px_bdg} and Lemma \ref{L-R}, the It\^{o}-Clifford integral can be defined for any element of $\mathcal{H}^p([0,t])$, and Burkholder-Gundy inequality holds ture.
For any $t\geq 0$, let $L^q_\mathbb{A}(0,t; L^p(\mathscr{C}_t))$
be the completion of ${\cal S}^p_{\mathbb A}([0,t])$ with the norm
$$
\|f\|_{L^q_\mathbb{A}(0,t; L^p(\mathscr{C}_t))}=\left(\int_0^t\|f(s)\|_p^qds\right)^{\frac{1}{q}}.
$$
Similarly, $L^p_\mathbb{A}(\mathscr{C}_t; L^q(0,t))$ is the completion of ${\cal S}^p_{\mathbb A}([0,t])$ with the norm
$$\|f\|_{L^p_\mathbb{A}(\mathscr{C}_t; L^q(0,t))}=\left\|\left(\int_0^t|f(s)|^qds\right)^{\frac{1}{q}}\right\|_p,\  t\geq 0.$$
As an application of Minkowski inequality,  we have the following result.
\begin{thm}\label{p-q exchange}  
Let $1\leq q\leq p<\infty$. Then, for any $f\in L^q_{\mathbb A}(0,T;L^p(\mathscr{C}_T))$,
\begin{equation}\label{exchange the norm}
\left\|\left(\int_0^t|f(s)|^qds\right)^{\frac{1}{q}}\right\|_p\leq  \left(\int_0^t\|f(s)\|_p^qds\right)^{\frac{1}{q}}, \  0\leq t\leq T.
\end{equation}
Furthermore, $L^q_\mathbb{A}(0,T; L^p(\mathscr{C}_T))\subseteq L^p_\mathbb{A}(\mathscr{C}_T; L^q(0,T))$.
\end{thm}
\begin{proof}
 Let
$$0=t_0\leq t_1\leq t_2\leq\cdots\leq t_n=t$$
be an equal time partition of $[0,t]$ where the mesh of the subdivision is $l=t/n=t_{k+1}-t_k$, $k=0,\ 1,\ \cdots,\ n-1$.
For simple adapted process $\sum\limits_{k\geq 0} a_{t_k}\chi_{[t_k,t_{k+1})}$ of $L^p(\mathscr{C})$
where $a_{t_k}\in L^p(\mathscr{C}_{t_k})$, and
 any positive integer $n$, one has
\begin{equation}\label{linear of norm}
 \left\|\left(\sum_{k=0}^{n-1}|a_{t_k}|^q(t_{k+1}-t_k)\right)^{\frac{1}{q}}\right\|_p
= l^{\frac{1}{q}}\left\|\left(\sum_{k=0}^{n-1}|a_{t_k}|^q\right)^{\frac{1}{q}}\right\|_p
 =l^{\frac{1}{q}}\left\|\sum_{k=0}^{n-1}|a_{t_k}|^q\right\|_{\frac{p}{q}}^{\frac{1}{q}}.
\end{equation}
Since 
$\frac{p}{q}\geq1$, by Minkowski inequality \cite[Theorem 5.2.2]{XZC},
\begin{equation}\label{M-inequality}
\left\|\sum_{k=0}^{n-1}|a_{t_k}|^q\right\|_{\frac{p}{q}}\leq\sum_{k=0}^{n-1}\left\||a_{t_k}|^q\right\|_{\frac{p}{q}}, \   n\in\mathbb{N}^+.
\end{equation}
By \eqref{linear of norm} and \eqref{M-inequality}, we have
\begin{equation*}
\begin{aligned}
&\left\|\left(\sum_{k=0}^{n-1}|a_{t_k}|^q(t_{k+1}-t_k)\right)^{\frac{1}{q}}\right\|_p
\leq
\left(\sum_{k=0}^{n-1}\|a_{t_k}\|_{p}^q(t_{k+1}-t_k)\right)^{\frac{1}{q}}.
\end{aligned}
\end{equation*}
Since ${\cal S}^p_{\mathbb A}([0,T])$ is dense in $L^p_\mathbb{A}(\mathscr{C}; L^q(0,T))$,
$$
\left\|\left(\int_0^t|f(s)|^qds\right)^{\frac{1}{q}}\right\|_p
\leq\left(\int_0^t\|f(s)\|^q_pds\right)^{\frac{1}{q}}, \   0\leq t\leq T,
$$
for any $f\in L^q_\mathbb{A}(0,T; L^p(\mathscr{C}_T))$, and $L^q_\mathbb{A}(0,T; L^p(\mathscr{C}_T))\subseteq L^p_\mathbb{A}(\mathscr{C}_T; L^q(0,T))$.
\end{proof}
\begin{rem}
Actually,  the above inequality \eqref{exchange the norm} holds for any $0< q\leq p<\infty$.
In particular, when p=q, the equality in the above inequality (2.5) holds and is Fubini's theorem in the noncommutative case. 
\end{rem}

In general, the Burkholder-Gundy inequality \eqref{BDG-PX} cannot be directly used to study $L^p$-solution to QSDE.
Instead, we need the following easy corollary.
 \begin{cor}\label{cor-bp}Let $p>2$. Then,
 for any $f\in L^2_{\mathbb A}(0,T; L^p(\mathscr{C}_T))$, there is positive constant $C(p)$ such that
 \begin{equation}\label{bp}
\left\|\int_0^tf(s)dW_s\right\|_p\leq C(p) \left(\int_0^t\|f(s)\|^2_pds\right)^{\frac{1}{2}},\  \ 0\leq t\leq T.
\end{equation}
Moreover, $L^2_{\mathbb A}(0,T; L^p(\mathscr{C}_T))\subseteq {\cal H}^p([0,T])$ and
\begin{equation}\label{bp1}
\|f\|_{{\cal H}^p([0,t])}\leq  \left(\int_0^t\|f(s)\|^2_pds\right)^{\frac{1}{2}}, \  \ 0\leq t\leq T.
\end{equation}
 \end{cor}
 \begin{proof}According to Theorem \ref{p-q exchange}, for any $f\in L^2_{\mathbb A}(0,T; L^p(\mathscr{C}_T))$, one has
 \begin{equation}
\left\|\left(\int_0^t|f(s)|^2ds\right)^{\frac{1}{2}}\right\|_p\leq \left(\int_0^t\|f(s)\|_p^2ds\right)^{\frac{1}{2}}.
\end{equation}
Since $\|f(s)\|_p=\|f^*(s)\|_p$ for any $s\in[0,T]$,
$$
\|f\|_{{\cal H}^p([0,t])}=\max\left\{\left\|\left(\int_0^t|f(s)|^2ds\right)^{\frac{1}{2}}\right\|_p,\ \left\|\left(\int_0^t|f(s)^*|^2ds\right)^{\frac{1}{2}}\right\|_p\right\}\leq  \left(\int_0^t\|f(s)\|^2_pds\right)^{\frac{1}{2}}.
$$
Combined with the Burkholder-Gundy inequality \eqref{BDG-PX}, we have \eqref{bp} immediately.
\end{proof}
Now, we give the parity of each element of $L^p(\mathscr{C})$. Let the parity operator $P$ be automorphism map on von Neumann algebra $\mathscr{C}$ generated by bounded linear operators on $\Lambda(\mathscr{H})$ as is in \cite{BSW1,BSW3,PX}.
 \begin{defn}
For any $h\in L^p(\mathscr{C})$, $h$ is said to be odd if $Ph=-h$, $h$ is said to be even if $P h=h$.
And, $h$ has definite parity if $h$ is even or odd.
\end{defn}
Furthermore, for any $1< p<\infty$,
\begin{equation}\label{decomposition of even and odd}
L^p(\mathscr{C})=L^p(\mathscr{C}_o)\oplus L^p(\mathscr{C}_e),
\end{equation}
where $L^p(\mathscr{C}_e),\ L^p(\mathscr{C}_o)$ denote the even part and the odd part, respectively. More precisely, for any $h\in L^p(\mathscr{C})$, $h=\frac{h+Ph}{2}+\frac{h-Ph}{2}=h_o+h_e$, where $h_e$ and $h_o$ are even and odd, respectively. 
Since $P$ is isometric on $L^p(\mathscr{C})$,
\begin{equation*}
\max\left\{\|h_o\|_p, \|h_e\|_p\right\}\leq \|h\|_p\leq \|h_o\|_p+\|h_e\|_p.
\end{equation*}

Let $\mathscr{E}$ denote the algebra of even polynomials in the fields $\{\Psi(u):u\in\mathscr{H}\}$,
and let $\mathscr{C}_e$ be the $W^{*}$-subalgebra of $\mathscr{C}$ generated by $\mathscr{E}$.
If $h\in L^p(\mathscr{C})$ is even there is a sequence $\{h_n\}$ in $\mathscr{E}$ such that $h_n\rightarrow h$ in $L^p(\mathscr{C})$, and therefore $h_n^{*}\rightarrow h^{*}$ in $L^p(\mathscr{C})$. It follows that $h^{*}$ is also even. Similarly, if $g$ is odd in $L^p(\mathscr{C})$, there is a sequence $\{g_n\}$ of odd polynomials in the fields with $g_n\rightarrow g$ and thus $g_n^{*}\rightarrow g^{*}$ in $L^p(\mathscr{C})$, that is, $g^{*}$ is odd as well. It follows that if $h=h^{*}$ in $L^p(\mathscr{C})$ and $h=h_e+h_o$, then $h_e=h_e^{*}$ and $h_o=h_o^{*}$ in $L^p(\mathscr{C})$.

\begin{lem}\label{exchange-Brownian-definite parity}\cite[Lemma 3.15]{BSW1}
Let $\{W_t\}_{t\geq t_0}$ be Brownian motion. If $h\in L^p(\mathscr{C}_{t_0})$ has definite parity, then
$$h(W_{t_2}-W_{t_1})=\pm(W_{t_2}-W_{t_1})h,\quad t_0\leq t_1\leq t_2,
$$
depending on whether $h$ is even or odd.
\end{lem}

\begin{lem}\cite[Theorem 1.8.2 Bihari inequality]{M}\label{Bihari's inequality} Let $\rho:[0,+\infty)\to[0,+\infty)$ be a continuous and non-decreasing function vanishing at 0 satisfying $\int_{0^+}\frac{dr}{\rho(r)}=\infty$. Suppose $u(t)$ is a continuous nonnegative function on $[t_0,T]$ such that
\begin{equation}\label{B-inequality-1}
u(t)\leq u_0+\int_{t_0}^t\phi(r)\rho(u(r))dr, \  t_0\leq t\leq T,
\end{equation}
where $u_0$ is a nonnegative cinstant and $\phi:[t_0,T]\to \mathbb{R}^+$, then
$$u(t)\leq U^{-1}\left(U(u_0)+\int_{t_0}^t\phi(r)dr\right), \  t_0\leq t\leq T,
$$
where $U(t)=\int_{t_0}^t\frac{1}{\rho(r)}dr$, $U^{-1}$ is the conver function of $U$. In particular, $u_0=0$, then $u(t)=0$ for all $t_0\leq t\leq T$.
\end{lem}
\section{The existence and uniqueness of solutions to QSDE}
\indent\indent
This section is devoted to proving the existence and uniqueness of $L^p$-solution to the equation \eqref{nonlocal condition-QSDE} with non-Lipschitz coefficients for $p>2$.

\begin{thm}\label{E-U}
Let Assumption \ref{Assump} hold. Then 
 the equation \eqref{nonlocal condition-QSDE} admits a unique solution $X(\cdot)\in C_{\mathbb{A}}(t_0,T;L^p(\mathscr{C}_T))$.
\end{thm}
\begin{proof}
We shall deal with the existence and uniqueness separately.\\
{\bf{Existence:}} The proof of the existence is divided into three steps.

\textbf{Step 1.} The iteration $\{X_t^{(n)}\}_{n\geq 0}$ is well-defined for any $t\in[t_0,T]$.
Let $T>t_{0}$,  $t_0\leq t\leq T$ be fixed. For any non-negative integer $n$, define $X_t^{(n)}$ in $L^p(\mathscr{C})$ inductively by
\begin{equation}\label{Iteration}
X_t^{(n+1)}=Z+R(X_t^{(n+1)})+\int_{t_0}^{t}F(X_s^{(n)},s)dW_s+\int_{t_0}^{t}dW_s G(X_s^{(n)},s)+\int_{t_0}^{t}H(X_s^{(n)},s)ds,
\end{equation}
where the well-definedness of iteration comes from Banach fixed point theorem and that $R$ is a strict contraction.

Firstly, we claim that each $X_t^{(n)}$, $n\geq1$, defines an adapted $L^p$-continuous process on $[t_0,T]$ by induction. By assumption, $F(Z,s),\ G(Z, s)$ and $H(Z,s)$ are $L^p$-continuous with respect to $s$ and belong to $L^p(\mathscr{C}_s)$ for $t_0\leq s\leq T$, then quantum stochastic integral $X_t^{(1)}$ exists for $t\in[t_0, T]$.
Furthermore, we can obtain the boundedness of $X_t^{(1)}$ by the continuity on compact sets and easily verify that $t\mapsto X_t^{(1)}$ is continuous:~$[t_0, T]\rightarrow L^p(\mathscr{C})$.

Now, if $X_t^{(n)}$ is assumed to be adapted and continuous, then $F(X_t^{(n)},t),~G(X_t^{(n)},t)$ and $H(X_t^{(n)},t)$ are adapted, $L^p$-continuous on $[t_0, ~T]$ and bounded, thus $X_t^{(n+1)}$ is adapted.
For any $t_1,\ t_2\in [t_0,T]$, by \eqref{BDG-PX} and Assumption \ref{Assump},
\begin{align*}
\|X_{t_1}^{(n+1)}-X_{t_2}^{(n+1)}\|_p&\leq \|R(X_{t_1}^{(n+1)})-R(X_{t_2}^{(n+1)})\|_p+\left\|\int_{t_1}^{t_2}F(X_s^{(n)},s)dW_s\right\|_p\\
&\indent+\left\|\int_{t_1}^{t_2}dW_s G(X_s^{(n)},s)\right\|_p+\left\|\int_{t_1}^{t_2}H(X_s^{(n)},s)ds\right\|_p\\
&\leq C(R)\|X_{t_1}^{(n+1)}-X_{t_2}^{(n+1)}\|_p+C(p)\|F(X_s^{(n)},s)\|_{\mathcal{H}^p([t_1,t_2])}\\
&\indent+C(p)\|G(X_s^{(n)},s)\|_{\mathcal{H}^p([t_1,t_2])}+\left\|\int_{t_1}^{t_2}H(X_s^{(n)},s)ds\right\|_p.
\end{align*}
Now, subtracting $C(R)\|X_{t_1}^{(n+1)}-X_{t_2}^{(n+1)}\|_p$ from both sides of above inequality and applying Corollary \ref{cor-bp} and H\"{o}lder inequality, we get
\begin{align*}
&(1-C(R))\|X_{t_1}^{(n+1)}-X_{t_2}^{(n+1)}\|_p\\
&\leq C(p)\left(\int_{t_1}^{t_2} \|F(X_s^{(n)},s)\|_p^2ds\right)^{\frac{1}{2}}+C(p)\left(\int_{t_1}^{t_2}\|G(X_s^{(n)},s)\|_p^2ds\right)^{\frac{1}{2}}+C(T)\left(\int_{t_1}^{t_2}\|H(X_s^{(n)},s)\|_p ^2ds\right)^{\frac{1}{2}},
\end{align*}
which implies that $t\mapsto X_t^{(n+1)}$ is $L^p$-continuous on $[t_0, T]$. 
Hence we have proved our claim by induction.

\textbf{Step 2.} The sequence of iteration is convergent under the given conditions. For any $t\in[t_0, T]$, by Minkowski inequality,
$$
\begin{aligned}
&\|X_t^{(n+1)}-X_t^{(n)}\|_p\\
&\indent\leq\|R(X_t^{(n+1)})-R(X_t^{(n)})\|_p+\left\|\int_{t_0}^t(F(X_s^{(n)},s)-F(X_s^{(n-1)},s))dW_s\right\|_p\\
&\indent\indent+\left\|\int_{t_0}^tdW_s(G(X_s^{(n)},s)-G(X_s^{(n-1)},s))\right\|_p+\left\|\int_{t_0}^t(H(X_s^{(n)},s)-H(X_s^{(n-1)},s))ds\right\|_p.
\end{aligned}
$$
By similar analysis as above, the elementary inequality $(a+b+c)^2\leq3(a^2+b^2+c^2)$, H\"{o}lder inequality and 
 Osgood conditions of $F,\ G,\ H$,
there is constant $C( p, R, T)$ such that
\begin{equation*}
\begin{aligned}
&\|X_t^{(n+1)}-X_t^{(n)}\|_p^2\\
&\indent\leq\frac{3}{(1-C(R))^2}\Bigg(C^2(p) \int_{t_0}^t\|F(X_s^{(n)},s)-F(X_s^{(n-1)},s)\|_p^2ds\\
&\indent\indent+C^2(p)\int_{t_0}^t\|G(X_s^{(n)},s)-G(X_s^{(n-1)},s)\|_p^2ds+C^2(T)\int_{t_0}^t\|H(X_s^{(n)},s)-H(X_s^{(n-1)},s)\|_p^2ds\Bigg)\\
&\indent\leq C(p, R, T)\int_{t_0}^{t}\Big(\|F(X_s^{(n)},s)-F(X_s^{(n-1)},s)\|_p^2+\|G(X_s^{(n)},s)-G(X_s^{(n-1)},s)\|_p^2\\
&\indent\indent+\|H(X_s^{(n)},s)-H(X_s^{(n-1)},s)\|_p^2\Big)ds\\
&\indent\leq C( p, R, T)\int_{t_0}^{t}\rho\left(\|X_s^{(n)}-X_s^{(n-1)}\|_p^2\right)ds,
\end{aligned}
\end{equation*}
where $C(p, R, T)=\frac{3}{(1-C(R))^2}\max\{C^2(p),\  C^2(T)\}$. 

Therefore, for any $n,\ k\geq 1$, $t\in[t_0,T]$,
$$
\|X_t^{(n+k)}-X_t^{(n)}\|_p^2\leq C(p, R, T)\int_{t_0}^t\rho\left(\|X_s^{(n+k-1)}-X_s^{(n-1)}\|_p^2\right)ds.
$$
Since each $X_t^{(n)}$ is $L^p$-continuous process on $[t_0,T]$ for any $n\in \mathds{N}^+$, $\|X_t^{(n)}\|_p$ is uniformly bounded on $[t_0,T]$. Set $u_{n,k}(t)=\sup\limits_{s\in[t_0,t]}\|X_s^{(n+k)}-X_s^{(n)}\|_p^2$, $t\in[t_0,T]$, which is uniformly bounded. Then
$$u_{n,k}(t)\leq C(p, R, T)\int_{t_0}^t\rho(u_{n-1,k}(s))ds. $$
Let $v_n(t)=\sup\limits_ku_{n,k}(t)$, $t\in[t_0, T]$. Then,
$$0\leq v_n(t)\leq C(p, R, T)\int_{t_0}^t\rho(v_{n-1}(s))ds.$$
Denote $$\alpha(t)=\limsup_{n\to+\infty} v_n(t),\  t_0\leq t\leq T.$$
Applying Lebesgue dominated convergence theorem, we get
$$0\leq \alpha(t)\leq C(p, R, T)\int_{t_0}^t\rho(\alpha(s))ds,\  t_0\leq t\leq T.$$
Hence, by Lemma \ref{Bihari's inequality}, one deduces
$$\alpha(t)=0,\ t_0\leq t\leq T,$$
which implies that $\{X_t^{(n)}\}_{n\geq0}$ is a Cauchy sequence in $L^p(\mathscr{C})$.

\textbf{Step 3.} $X(\cdot)\in C_\mathbb{A}(t_0,T;L^p(\mathscr{C}_T))$ is the solution to QSDE \eqref{nonlocal condition-QSDE}.
Since $\{X_t^{(n)}\}_{n\geq0}$ is a Cauchy sequence in $L^p(\mathscr{C})$, there exists $X_t\in L^p(\mathscr{C})$ such that for any $t\in[t_0,T]$,
$$\lim_{n\to\infty}\|X^{(n)}_t-X_t\|_p=0.$$
Thus, for any $\varepsilon>0$, there exists $\delta>0$ such that 
\begin{align*}
\|X_{t_1}-X_{t_2}\|_p&=\|X_{t_1}-X_{t_1}^{(n)}+X_{t_1}^{(n)}-X_{t_2}^{(n)}+X_{t_2}^{(n)}-X_{t_2}\|_p\\
&\leq\|X_{t_1}-X_{t_1}^{(n)}\|_p+\|X_{t_1}^{(n)}-X_{t_2}^{(n)}\|_p+\|X_{t_2}^{(n)}-X_{t_2}\|_p\\
&<\varepsilon\quad \textrm{as}\ n\to\infty,\ \forall\ t_1, t_2\in [t_0,T]\  \textrm{satisfying}\ |t_1-t_2|<\delta.
\end{align*}
It shows that $X_t$ is $L^p$-continuous and adapted on $[t_0,T]$ since $X_t^{(n)}$ is $L^p$-continuous and adapted.

We shall prove that $\{X_t\}_{t\geq t_0}$ is the solution to
$$
X_t=Z+R(X_t)+\int_{t_0}^tF(X_s,s)dW_s+\int_{t_0}^t dW_s G(X_s,s)+\int_{t_0}^tH(X_s,s)ds,\ \textrm{a.s.\ in}\ [t_0,T].
$$
In fact,
$$
\begin{aligned}
\left\|\int_{t_0}^t F(X_s^{(n)},s)dW_s-\int_{t_0}^t F(X_s,s)dW_s\right\|_p^2&\leq C^2(p)\|F(X_s^{(n)},s)-F(X_s,s)\|^2_{\mathcal{H}^{p}([t_0,t])}\\
&\leq C^2(p)\int_{t_0}^t\|F(X_s^{(n)},s)-F(X_s,s)\|_p^2ds\\
&\leq C^2(p)\int_{t_0}^t\rho\left(\|X_s^{(n)}-X_s\|_p^2\right)ds,\\
&\rightarrow 0,\ \textrm{as}\ n\to \infty,
\end{aligned}
$$
since $X_s^{(n)}\to X_s$ in $L^p(\mathscr{C})$ for any $s\in[t_0,T]$ and $\rho$ is continuous. 
Similarly,
$$
\int_{t_0}^t dW_s G(X_s^{(n)},s)\rightarrow\int_{t_0}^t dW_s G(X_s,s)~{\rm and }~ \int_{t_0}^tH(X_s^{(n)},s)ds\rightarrow \int_{t_0}^tH(X_s,s)ds
$$
in $L^p(\mathscr{C})$. 
Because $X_s^{(n)}\rightarrow X_s$ for any $s\in[t_0,T]$, the same is true for $R(X_s^{(n)})\rightarrow R(X_s)$.\\
Taking limits on both sides of \eqref{Iteration}, it deduces that
$$
\begin{aligned}
X_t=&\lim_{n\rightarrow\infty}X_t^{(n+1)}\\
=&\lim_{n\rightarrow\infty}\left(Z+R(X^{(n+1)}_t)+\int_{t_0}^tF(X_s^{(n)},s)dW_s+\int_{t_0}^t dW_s G(X_s^{(n)},s)+\int_{t_0}^tH(X_s^{(n)},s)ds\right)\\
=&Z+R(X_t)+\int_{t_0}^tF(X_s,s)dW_s+\int_{t_0}^t dW_s G(X_s,s)+\int_{t_0}^tH(X_s,s)ds, \ \textrm{in}\ [t_0,T].
\end{aligned}
$$
That is, $\{X_t\}_{t\geq t_0 }$ is a $L^p$-solution to the equation \eqref{nonlocal condition-QSDE}. 

\textbf{Uniqueness:}
Suppose that $Y_t,\ t\in [t_0,T]$ is another adapted $L^p$-continuous solution with $Y_{t_0}=Z+R(Y)$.  Then, by \eqref{nonlocal condition-QSDE}, we obtain again
$$Y_t=Z+R(Y_t)+\int_{t_0}^tF(Y_s,s)dW_s+\int_{t_0}^tdW_s G(Y_s,s)+\int_{t_0}^tH(Y_s,s)ds,\ \textrm{a.s.\ in}\ [t_0,T].$$
Furthermore,
\begin{align*}
\|X_t-Y_t\|_p&\leq\|R(X_t)-R(Y_t)\|_p+\left\|\int_{t_0}^t(F(X_s,s)-F(Y_s,s))dW_s\right\|_p\\
&\indent+\left\|\int_{t_0}^t dW_s(G(X_s,s)-G(Y_s,s))\right\|_p+\left\|\int_{t_0}^t(H(X_s,s)-H(Y_s,s))ds\right\|_p.
\end{align*}
Continuing to use the same technique as Step 2 of existence, we can yield that
\begin{equation*}
\|X_t-Y_t\|_p^2\leq C(p, R, T)\int_{t_0}^t\rho\left(\|X_s-Y_s\|_p^2\right)ds,\  t_0\leq t\leq T.
\end{equation*}
It follows that, for any $t\in[t_0, T]$,
$$\|X_t-Y_t\|_p=0,\ \textrm{ a.s.}$$
This completes the proof.
\end{proof}
As described in \cite{AH1}, the It\^{o} product rule $dA(\chi_{[0,t)})dA^{*}(\chi_{[0,t)})=dt$ holds for any $t\geq 0$. Based on \cite{BSW4}, let 
$$\xi_t=\alpha_1 A(u\chi_{[0,t)})+\alpha_2 A^*(u\chi_{[0,t)})$$for any $t\geq 0$, the integral $\int_0^tf(s)d\xi_s$ defines a quantum martingale for any $f\in {\cal S}_{\mathbb A}^p(\mathbb R^+)$.
Next, let $A(t):=A( \chi_{[0,t)})$, we study the properties of the $L^p$-solutions to QSDE \eqref{SDE-2} with respect to Brownian motion $A(t)$ and $A^*(t)$
on the basis of martingale inequalities.
From Lemma \ref{exchange-Brownian-definite parity} and the canonical anticommutation relation, we can deduce the following martingale inequalities.
\begin{thm}\label{estimate-A-A*-martingale}
Let $f:\ [0,T]\to L^p(\mathscr{C})$ be adapted processes with $p\geq 2$. Then, for any $t\in[0,T]$, $\int_0^tf(s)dA(s)$ and $\int_0^tdA^*(s)f(s)$ are $L^p$-martingales and
\begin{equation}\label{estimate-A-A*}
 \begin{aligned}
\left\|\int_0^tf(s)dA(s)\right\|_p\leq \beta_p \left(\int_0^t\|f(s)\|_p^2ds\right)^{\frac{1}{2}},\\
\left\|\int_0^t dA^*(s)f(s)\right\|_p\leq \beta_p \left(\int_0^t\|f(s)\|_p^2ds\right)^{\frac{1}{2}}.
 \end{aligned}
\end{equation}
\begin{proof}
First, we consider simple adapted $L^p$-process $f\in {\cal S}^p_{\mathbb A}([0,T])$, then $\int_0^tf(s)dA(s)$ and $\int_0^tdA^*(s)f(s)$ are $L^p$-martingales.

Let$$0=t_0\leq t_1\leq t_2\leq \ldots\leq t_n=t$$
be a partition of $[0,t]$.
Then
$$Q(t)=\int_0^tf(s)dA(s)=\sum_{k=0}^{n-1}f(t_k)\left(A(t_{k+1})-A(t_k)\right),$$
$$Q(t_k)=\sum_{i=0}^{k-1}f(t_i)\left(A(t_{i+1})-A(t_i)\right), \  k\in \mathds{N}.$$
Define the martingale difference of $Q(t)$ as
$$dQ_k=Q(t_{k+1})-Q(t_{k})=f(t_k)(A(t_{k+1})-A(t_{k})), \ k\in \mathds{N}^{+}.$$
By Theorem 2.1 of \cite{PX}, there exists a positive constant $\beta_p$ such that
\begin{equation}\label{A-A*-BG}
\left\|\int_0^tf(s)dA(s)\right\|_p\leq \beta_p\max\left\{\left\|\left(\sum_{k\geq0}|dQ_k|^2\right)^{\frac{1}{2}}\right\|_p, \quad \left\|\left(\sum_{k\geq0}|dQ^*_k|^2\right)^{\frac{1}{2}}\right\|_p\right\}.
\end{equation}
By the canonical anticommutation relation $$A(t)A^*(t)+A^*(t)A(t)=t,\ t\geq 0,$$
one has
\begin{equation}\label{CAR}
  (A(t)-A(s))(A^*(t)-A^*(s))\leq t-s,\ (A^*(t)-A^*(s))(A(t)-A(s))\leq t-s,\  0\leq s\leq t\leq T.
\end{equation}
According to \eqref{decomposition of even and odd}, $f=f_e+f_o$ for any $f\in L^p(\mathscr{C})$, then
\begin{equation}\label{Q-Estimate}
  \begin{aligned}
     \sum_{k\geq0}|dQ_k|^2 =&\sum_{k=0}^{n-1}\left(A^*(t_{k+1})-A^*(t_{k})\right)f^*(t_k)f(t_k)\left(A(t_{k+1})-A(t_{k})\right) \\
     =&\sum_{k=0}^{n-1}\left(f^*_e(t_k)-f^*_o(t_k)\right)\left(A^*(t_{k+1})-A^*(t_{k})\right)\left(A(t_{k+1})-A(t_{k})\right)\left(f_e(t_k)-f_o(t_k)\right)\\
    \leq&\sum_{k=0}^{n-1}\left(f^*_e(t_k)-f^*_o(t_k)\right)(t_{k+1}-t_k)\left(f_e(t_k)-f_o(t_k)\right)\\
    =&\int_0^t|f_e(s)-f_o(s)|^2ds,
   \end{aligned}
\end{equation}
and
\begin{equation}\label{Q*-Estimate}
  \begin{aligned}
     \sum_{k\geq0}|dQ^*_k|^2 =&\sum_{k=0}^{n-1}f(t_k)(A(t_{k+1})-A(t_{k}))(A^*(t_{k+1})-A^*(t_{k}))f^*(t_k) \\
    \leq& \sum_{k=0}^{n-1}f(t_k)f^*(t_k)(t_{k+1}-t_{k})\\
    =&\int_0^t|f^*(s)|^2ds,
   \end{aligned}
\end{equation}
where the above two inequalities are based on Lemma \ref{exchange-Brownian-definite parity} and \eqref{CAR}.

Substituting \eqref{Q-Estimate} and \eqref{Q*-Estimate} into the right side of \eqref{A-A*-BG} and applying Theorem \ref{p-q exchange}, we get
\begin{equation*}
\begin{aligned}
\left\|\int_0^tf(s)dA(s)\right\|_p&\leq \beta_p\max\left\{\left\|\left(\int_0^t|f^*(s)|^2ds\right)^\frac{1}{2}\right\|_p,\ \left\|\left(\int_0^t|f_e(s)-f_o(s)|^2ds\right)^{\frac{1}{2}}\right\|_p\right\}\\
&\leq \beta_p\max\left\{\left(\int_0^t\|f^*(s)\|_p^2ds\right)^\frac{1}{2},\ \left(\int_0^t\|f_e(s)-f_o(s)\|_p^2ds\right)^\frac{1}{2} \right\}.
\end{aligned}
\end{equation*}
On the other hand, for any $f\in {\cal S}^p_{\mathbb A}([0,T])$ and any  $s\in[0,T]$,
$$\|f_e(s)-f_o(s)\|_p\leq2\|f(s)\|_p,\ \|f^*(s)\|_p=\|f(s)\|_p.$$
Thus we have
\begin{equation}\label{estimate-A}
\left\|\int_0^tf(s)dA(s)\right\|_p\leq \beta_p\left(\int_0^t\|f(s)\|_p^2ds\right)^\frac{1}{2},\ 0\leq t\leq T.
\end{equation}
Similarly, one has
\begin{equation}\label{estimate-A*}
\left\|\int_0^t dA^*(s)f(s)\right\|_p\leq \beta_p\left(\int_0^t\|f(s)\|_p^2ds\right)^\frac{1}{2},\  0\leq t\leq T.
\end{equation}

Finally, since the general adapted $L^p$-processes can be approximated by simple processes, \eqref{estimate-A-A*} can be directly obtained from \eqref{estimate-A} and \eqref{estimate-A*}.
\end{proof}
\end{thm}
Clearly, by virtue of Theorem \ref{E-U} and Theorem \ref{estimate-A-A*-martingale}, the following result holds.
\begin{cor}\label{E-U-A-A*}
Let Assumption \ref{Assump} hold. 
Then there is a unique solution $X(\cdot)\in C_\mathbb{A}(t_0,T;L^p(\mathscr{C}_T))$ to the equation \eqref{SDE-2} 
with nonlocal condition $X_{t_0}=Z+R(X)\in L^p(\mathscr{C}_{t_0})$ on $[t_0,T]$.
\end{cor}

\section{The stability of solutions to QSDE with Lipschitz condition}
\indent\indent
In this section, we shall prove that the $L^p$-solution to the equation \eqref{nonlocal condition-QSDE} is stable, namely, small changes in the initial condition and in the coefficients $F,\ G,\ H$ and $R$  lead to small changes in the solution on $[t_0,T]$.

Let the coefficients $F,\ G$, $H$ of the equation \eqref{nonlocal condition-QSDE} satisfy Lipsctitz condition, i.e.\\
\textbf{(A3')} For any $x_1,\ x_2\in L^p(\mathscr{C})$ and a.e. $t\in [0,T]$, there exists a constant $L> 0$ such that
$$
   \|F(x_1, t)-F(x_2, t)\|_p^2+\|G(x_1, t)-G(x_2, t)\|_p^2+ \|H(x_1, t)-H(x_2, t)\|_p^2\leq L\|x_1-x_2\|_p^2 .$$

Let $\{X_t\}_{t\geq t_0}$, $\{Y_t\}_{t\geq t_0}$ be the $L^p$-solution to the equation \eqref{nonlocal condition-QSDE}
with initial conditions $X_{t_0}=Z+R(X)$ and $Y_{t_0}=Z^{'}+R(Y)$ for any $X_{t_0},\ Y_{t_0}\in L^p(\mathscr{C}_{t_0})$, respectively.
That is,
$$X_t=Z+R(X_t)+\int_{t_0}^tF(X_s,s)dW_s+\int_{t_0}^tdW_s G(X_s,s)+\int_{t_0}^tH(X_s,s)ds,\ \textrm{ a.s.\ in}\ [t_0,T],$$
and
$$Y_t=Z^\prime+R(Y_t)+\int_{t_0}^tF(Y_s,s)dW_s+\int_{t_0}^tdW_s G(Y_s,s)+\int_{t_0}^tH(Y_s,s)ds,\ \textrm{ a.s.\ in}\ [t_0,T].$$

\begin{thm}\label{Stability}
Suppose that assumptions (A1),(A2),(A3'),(A4) hold. With the above notations, for any $\varepsilon>0$, 
there exists $\delta>0$ such that if $\|Z-Z^{'}\|_p<\delta$, then $\|X_t-Y_t\|_p<\varepsilon$ holds for all $t_0\leq t\leq T$.
\end{thm}
\begin{proof}
By the directly calculation,
\begin{align*}
\|X_t-Y_t\|_p&\leq\|Z-Z^\prime\|_p+\|R(X_t)-R(Y_t)\|_p+\left\|\int_{t_0}^t(F(X_s,s)-F(Y_s,s))dW_s\right\|_p\\
&\indent+\left\|\int_{t_0}^t dW_s(G(X_s,s)-G(Y_s,s))\right\|_p+\left\|\int_{t_0}^t(H(X_s,s)-H(Y_s,s))ds\right\|_p.
\end{align*}
According to the proof of Theorem \ref{E-U} again and $(A3')$, for any $t\in[t_0, T]$, one gains the following estimate
\begin{align*}
\|X_t-Y_t\|_p^2\leq &\frac{4}{(1-C(R))^2}\|Z-Z^{'}\|_p^2+C'(p,R,T)\Bigg(\int_{t_0}^t\|F(X_s,s)-F(Y_s,s)\|_p^2ds\\
&+\int_{t_0}^t\|G(X_s,s)-G(Y_s,s)\|_p^2ds+\int_{t_0}^t\|H(X_s,s)-H(Y_s,s)\|_p^2ds\Bigg)\\
\leq&\frac{4}{(1-C(R))^2}\|Z-Z^{'}\|_p^2+C(p,T,R,L)\int_{t_0}^t\|X_s-Y_s\|_p^2ds,
\end{align*}
where $C(p,T,R,L)=C'(p,R,T)L=\frac{4}{(1-C(R))^2}\max\{C^2(p), C^2(T)\}L$.

By Gronwall's inequality,
$$
\|X_t-Y_t\|_p^2\leq \frac{4}{(1-C(R))^2} e^{C(p,T,R,L)(t-t_0)}\|Z-Z^{'}\|_p^2,
$$
for all $t\in[t_0, T]$, and the desired result is obtained. 
\end{proof}

In a similar manner, we establish stability theorems for the connection between coefficient convergence and solution convergence under Lipschitz condition.
\begin{thm}\label{Stable-coe}
Let assumptions (A1),(A2),(A3'),(A4) hold with $F,\ G,\ H,\ R$ being replaced respectively by $F_n,\ G_n$, $ H_n,\ R_n$, for all $n=1,\ 2, \cdots$ and $W_t$ be as in the equation \eqref{nonlocal condition-QSDE}.
Assume that
$F_n\rightarrow F,\ G_n\rightarrow G,\ H_n\rightarrow H$ in $L^p(\mathscr{C})$ as $n\to\infty$, uniformly on $L^p(\mathscr{C})\times[t_0,T]$, $R_n\rightarrow R$ in $L^p(\mathscr{C})$ uniformly as $n\to\infty$ on $L^p(\mathscr{C})$, and $Z_n\rightarrow Z$ in $L^p(\mathscr{C}_{t_0})$.
Furthermore, let $X(\cdot),\ X_n(\cdot)\in C_\mathbb{A}(t_0,T; L^p(\mathscr{C}_T))$ be solutions to the equation \eqref{nonlocal condition-QSDE} corresponding to $F,\ G,\ H,\ R,\ Z$ and $F_n,\ G_n,\ H_n,\ R_n,\ Z_n$, respectively. Then $X_n(t)\rightarrow X(t)$ in $L^p(\mathscr{C})$ uniformly on compact set $[t_0,T]$.
\end{thm}
Likewise, by Lemma \ref{estimate-A-A*-martingale} and Corollary \ref{E-U-A-A*}, we could get the following result.
\begin{cor}
Suppose that assumptions (A1),(A2),(A3'),(A4) hold. Then the solution $X(\cdot)\in C_\mathbb{A}(t_0,T;L^p(\mathscr{C}_T)$ to the equation \eqref{SDE-2} is stable on $[t_0,T]$ when nonlocal condition $X_{t_0}=Z+R(X)\in L^p(\mathscr{C}_{t_0})$ and the coefficients change slightly, respectively.
\end{cor}

\section{The Self-adjointness and Markov Property }
\indent\indent
In this section, we consider the self-adjointness and Markov property of $L^p$-solutions to QSDE \eqref{SDE} with nonlocal conditions under non-Lipschitz coefficients for $p>2$. 

According to the description of parity in Section 2, we get the following lemma.
\begin{lem}\label{S-A1}
Let 
$F:[0,t]\rightarrow L^p(\mathscr{C})$ be adapted and satisfy $\int_{0}^t\|F(s)\|_p^2ds<\infty$.
 Suppose further that $F(s)=F(s)^{*}\in L^p(\mathscr{C}_e)$ for each $s\in [0,t]$. Then
$\int_{0}^tF(s)dW_s$ is self-adjoint element of $L^p(\mathscr{C})$, and $\int_{0}^tF(s)dW_s=\int_{0}^tdW_s F(s)$.
\end{lem}
\begin{proof}
It is sufficient to consider the case that $F(s)$ is simple with values in $\mathscr{E}$ for any $s\in[0,t]$. Since $F$ is simple,
$F(s)=\sum\limits_{k=0}^{n-1}F(t_k)\chi_{[t_k,t_{k+1})}(s)$ and $$\int_{0}^tF(s)dW_s=\sum_{k=0}^{n-1}F(t_k)(W_{t_{k+1}}-W_{t_{k}}),$$
where $\{t_k\}_{k=0}^n$ is a partition of $[0,t]$.
On the other hand, $F(s)=F(s)^*$ and $W_s$ is hermitian,
$$
\begin{aligned}
\left(\int_{0}^tF(s)dW_s\right)^{*}&=\left(\sum_{k=0}^{n-1}F(t_k)(W_{t_{k+1}}-W_{t_{k}})\right)^{*}\\
&=\sum_{k=0}^{n-1}\left(W_{t_{k+1}}-W_{t_{k}}\right)^{*}F(t_k)^{*}\\
&=\sum_{k=0}^{n-1}\left(W_{t_{k+1}}-W_{t_{k}}\right)F(t_k)\\
&=\int_{0}^tdW_s F(s).
\end{aligned}
$$
By Lemma \ref{exchange-Brownian-definite parity},  
$$\left(W_{t_{k+1}}-W_{t_{k}}\right)F(t_k)=F(t_k)\left(W_{t_{k+1}}-W_{t_{k}}\right),\ \forall\ F(t_k)\in L^p(\mathscr{C}_e).$$
Then $\int_{0}^t dW_s F(s)=\int_{0}^tF(s)dW_s $.
Since $\int_{0}^t\|F(s)\|_p^2ds<\infty$, $\int_0^tF(s)dW_s$ is self-adjoint element in $L^p(\mathscr{C})$ by virtue of Corollary \ref{cor-bp}.
\end{proof}
Let $L^p(\mathscr{C})_{sa}$ denote the self-adjoint part of $L^p(\mathscr{C})$.
Let $F_i,\ G_i:\mathbb{R}\to \mathbb{R},$ for $i=1,2$, and satisfy $\int|t||\widehat{F_i}(t)|dt<\infty,\ \int|t||\widehat{G}_i(t)|dt<\infty$.
Suppose that $F_i,\ G_i:L^p(\mathscr{C})_{sa}\to L^p(\mathscr{C})$ are adapted and satisfy the Osgood condition of Assumption \ref{Assump} on $L^p(\mathscr{C})$, and each $F_i$ is an even function.  
Set $$\widetilde{F}_i(h)=F_i(h_o),\   \widetilde{G}_i(h)=G(h_e),\ \forall\ h\in L^p(\mathscr{C})_{sa}. $$ 
Evidently, $\widetilde{F}_i(h),\ \widetilde{G}_i(h)$ ($i=1,2$) are even by Lemma 4.1 of \cite{BSW2} for any $h\in L^p(\mathscr{C})_{sa}$.
Let $$\widetilde{\Phi}_i=\widetilde{F}_i+\widetilde{G}_i,\  i=1,2.$$
It can be seen that $\widetilde{\Phi}_i$ satisfies the Osgood conditions and maps self-adjoint elements of $L^p(\mathscr{C})$ into self-adjoint elements of $L^p(\mathscr{C}_e)$.
Then we obtain the self-adjointness of the solutions to QSDE with nonlocal conditions.
\begin{thm}\label{Self-adjointness}
Let $\widetilde{\Phi}_1,\ \widetilde{\Phi}_2$ be as above. Let $\widetilde{H}:\mathbb{R}\to \mathbb{R}$ satisfy $\int|t||\widehat{\widetilde{H}}(t)|dt<\infty$, and $\widetilde{H}:L^p(\mathscr{C})_{sa}\rightarrow L^p(\mathscr{C})$ be adapted and satisfy assumption $(A3)$ on $L^p(\mathscr{C})$ in Assumption \ref{Assump}. Furthermore, $\widetilde{R}:L^p(\mathscr{C})_{sa}\rightarrow L^p(\mathscr{C})_{sa}$ satisfies assumption $(A4)$  in Assumption \ref{Assump}. Then, for any $Z=Z^*$, there is a unique self-adjoint, adapted, $L^p$-continuous solution $\{X_t\}_{t\geq t_0}$ to the following QSDE
\begin{equation}\label{selfadjointness-QSDE}
dX_t=\widetilde{\Phi}_1(X_t)dW_t+dW_t\widetilde{\Phi}_2(X_t)+\widetilde{H}(X_t)dt
\end{equation}
on $[t_0,T]$ with $X_{t_0}=Z+\widetilde{R}(X)\in L^p(\mathscr{C}_{t_0})$ provided $0<C(\widetilde{R})<1$.
\end{thm}
\begin{proof}
Since $\widetilde{\Phi}_1,\ \widetilde{\Phi}_2,\ \widetilde{H}$ satisfy the Osgood condition and $\widetilde{R}$ satisfy the Lipschitz condition as in Assumption \ref{Assump},
it follows from Theorem \ref{E-U} that the equation \eqref{selfadjointness-QSDE} admits a unique solution $X(\cdot)\in C_\mathbb{A}(t_0,T;L^p(\mathscr{C}_T))$ such that
$$X_t=Z+\widetilde{R}(X_t)+\int_{t_0}^t\widetilde{\Phi}_1(X_s)dW_s+\int_{t_0}^{t}dW_s\widetilde{\Phi}_2(X_s)+\int_{t_0}^{t}\widetilde{H}(X_s)ds,\ \textrm{a.s.\ in}\ [t_0,T].$$

Next, it is enough to prove the self-adjointness of the solution to the equation \eqref{selfadjointness-QSDE}.
We can define the following equation inductively with $X_{t_0}=Z+\widetilde{R}(X)$,
\begin{equation}\label{self}
X_t^{(n+1)}=Z+\widetilde{R}(X_t^{(n+1)})+\int_{t_0}^t\widetilde{\Phi}_1(X_s^{(n)})dW_s+\int_{t_0}^{t}dW_s\widetilde{\Phi}_2(X_s^{(n)})+\int_{t_0}^{t}\widetilde{H}(X_s^{(n)})ds.
\end{equation}
To prove the self-adjointness of $X_t$, it is sufficient to show that  $X_t^{(n+1)}$ is self-adjoint by induction for all $n\geq0$. It is obvious that $X_t^{(1)}$ is self-adjoint since $Z=Z^*$. Assume that $X_t^{(n)}$ is self-adjoint, then $\widetilde{\Phi}_i(X_s^{(n)})\in L^p(\mathscr{C}_e)\cap L^p(\mathscr{C}_s)_{sa}$. By Lemma \ref{S-A1}, $\int_{t_0}^t\widetilde{\Phi}_1(X_s^{(n)})dW_s$ and $\int_{t_0}^{t}dW_s\widetilde{\Phi}_2(X_s^{(n)})$ are self-adjoint. In addition, $\int_{0}^{t}H(X^{(n)}_s)ds$ and $\widetilde{R}(X_s^{(n+1)})$ are also self-adjoint. Hence $X_t^{(n+1)}$ is self-adjoint.
\end{proof}

This result of self-adjoint of solutions is the basis of studying optimal control problem of QSDE.
Apart from this, we also obtain the following Markov property of solutions to the equation \eqref{nonlocal condition-QSDE} under non-Lipschitz coefficients consistent with Theorem 2.2, Corollary 2.3 and Corollary 2.4 of \cite{BSW3}. 

For any interval $\textit{I}\subseteq[t_0, \infty)$, let $\mathscr{A}_\textit{I}$ denote the $W^*$-algebra generated by $\mathds{1}$ and the solution $X_t$ to the equation \eqref{nonlocal condition-QSDE} for $t\in \textit{I}$, and write $\mathscr{A}_s$ for $\mathscr{A}_{[s,s]}$. Since the solution $X_t$ is adapted, i.e. $X_t\in L^p(\mathscr{C}_t)$ for all $t\geq t_0$, it follows that $\mathscr{A}_\textit{I}$ is a $W^*$-subalgebra of $\mathscr{C}_t$ whenever $\textit{I}\subseteq[t_0,t]$. Let $\mathscr{\tilde{A}}_\textit{I}=\mathscr{A}_\textit{I} \vee \beta(\mathscr{A}_\textit{I})$ be the $W^*$-subalgebra of $\mathscr{C}$ generated by $\mathscr{A}_\textit{I}$ and $\beta(\mathscr{A}_\textit{I})$. It is clear that $\beta(\mathscr{\tilde{A}}_\textit{I})=\mathscr{\tilde{A}}$ and $\mathscr{\tilde{A}}_s\subseteq \mathscr{C}_s$ for any $s\geq t_0$.

Next, we denote the algebra generated by field differences. Let $\mathscr{F}_s$ denote the $W^*$-subalgebra of $\mathscr{C}$ generated by the field differences $\{W_t-W_s:\ t_0 \leq s\leq t\}$, and $\mathscr{\tilde{A}}_s\vee \mathscr{F}_s$ be the $W^*$-subalgebra of $\mathscr{C}$ generated by $\mathscr{\tilde{A}}_s$ and $\mathscr{F}_s$. Thus, $\beta(\mathscr{\tilde{A}}_s\vee \mathscr{F}_s)=\mathscr{\tilde{A}}_s\vee \mathscr{F}_s$. Then, we get the following Markov property of the adapted solution $\{X_t\}_{t\geq t_0}$ to the equation \eqref{nonlocal condition-QSDE}.

\begin{thm}\label{Markov Theorem}
Let assumption \ref{Assump} hold and $\{X_t\}_{t\geq t_0}$ be an adapted, unique, continuous $L^p$-solution to the equation \eqref{nonlocal condition-QSDE}. Then $X_s\in L^p(\mathscr{\tilde{A}}_s\vee \mathscr{F}_s)$ for all $t_0\leq s\leq t$. Moreover, the process $\{X_t\}_{t\geq t_0}$ is a Markov process in the following sense: for any $s\geq t_0$ and $f\in L^{p}(\mathscr{\tilde{A}}_{[s,\infty)}),$ one has
\begin{equation}\label{Markov property}
 m(f|\mathscr{\tilde{A}}_{[t_0,s]})=m(f|\mathscr{\tilde{A}}_{s}),
\end{equation}
where $m(\cdot|\mathscr{B})$ denotes the conditional expectation with respect to the subalgebra $\mathscr{B}$ of $\mathscr{C}$.
\end{thm}
The proof of Theorem \ref{Markov Theorem} is similar to Theorem 2.2 of \cite{BSW3}.
Furthermore, the result for the solution  of the equation \eqref{SDE-2} also holds.
\section{Conclusion}
\indent\indent
Utilizing the Burkholder-Gundy inequality of Clifford $L^p$-martingale, we obtain the existence, uniqueness and stability of the solutions to QSDEs with nonlocal conditions for $p>2$. 
In addition, the acquisition of the self-adjoint solution pave the way for the next study on the optimal control problems of QSDE. 

\bibliographystyle{IEEEtran}

\end{document}